\documentclass[11pt,twoside]{article}
\usepackage{amsmath,amssymb,theorem}
\usepackage{graphicx}

\usepackage[dvips]{epsfig}

\numberwithin{equation}{section}
\newtheorem{theorem}{Theorem}

\newenvironment{proof}{{\bf Proof}.\ }{ \hfill $\square$}
\begin{document}

\title{Finite-time singularities in the dynamical evolution of contact lines}

\author{D.E. Pelinovsky and A.R. Giniyatullin\\
{\small Department of Mathematics, McMaster
University, Hamilton, Ontario, Canada, L8S 4K1}}

\date{\today}
\maketitle

\begin{abstract}
We study finite-time singularities in the linear advection--diffusion equation
with a variable speed on a semi-infinite line. The variable speed is determined
by an additional condition at the boundary, which models the dynamics of a contact line
of a hydrodynamic flow at a $180^{\circ}$ contact angle. Using apriori energy
estimates, we derive conditions on variable speed that guarantee that a sufficiently
smooth solution of the linear advection--diffusion equation blows up in a finite time.
Using the class of self-similar solutions to the linear advection--diffusion equation,
we find the blow-up rate of singularity formation. This blow-up rate does not agree with
previous numerical simulations of the model problem.
\end{abstract}

\section{Introduction}

Contact lines are defined by the intersection of the rigid and free boundaries of the flow.
Flows with the contact line at a $180^{\circ}$ contact angle were discussed in \cite{Benney,Dussan},
where corresponding solutions of the Navier--Stokes equations were shown to have no physical meanings.
Recently, a different approach based on the lubrication approximation and thin film equations
was developed by Benilov \& Vynnycky \cite{Benilov}.

As a particularly simple model for the flow shown on Figure \ref{fig-1}, the authors
of \cite{Benilov} derived the nonlinear advection--diffusion
equation for the free boundary $h(x,t)$ of the flow:
\begin{equation}
\label{model}
\frac{\partial h}{\partial t} + \frac{\partial}{\partial x} \left[
\frac{h^3}{3} \left( \alpha^3 \frac{\partial^3 h}{\partial x^3}
+ \frac{\partial h}{\partial x} \right) + (1 - V(t)) h \right] =
0, \quad x > 0, \;\; t > 0,
\end{equation}
where $\alpha$ is a numerical constant.
The contact line is fixed at $x = 0$ in the reference frame moving with the velocity $-V(t)$
and is defined by the boundary conditions $h|_{x = 0} = 1$ and $h_x |_{x = 0} = 0$.
The flux conservation gives the boundary condition for $h_{xxx} |_{x = 0} = -\frac{3}{2 \alpha^3}$.
For convenience, we can fix $\alpha^3 = 3$. Existence of weak
solutions of the thin-film equation (\ref{model}) for constant $V(t)$
and Neumann boundary conditions on a finite interval was recently
constructed by Chugunova {\em et al.} \cite{Chugunova1,Chugunova2}.
\begin{figure}[htbp]
\begin{center}
\includegraphics[height=5cm]{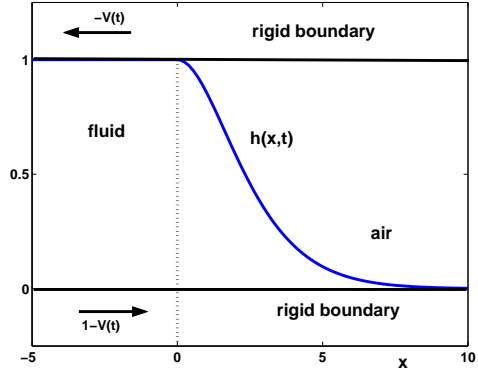}
\end{center}
\caption{Schematic picture of the flow between rigid boundaries.}
\label{fig-1}
\end{figure}

Using further asymptotic reductions with
\begin{equation}
\label{asymptotic-reduction}
h - 1 = \mathcal{O}(|V|^{-1}), \quad x = \mathcal{O}(|V|^{-1/3}), \quad
t = \mathcal{O}(|V|^{-4/3}), \quad \mbox{\rm as} \quad |V| \to \infty,
\end{equation}
the authors of \cite{Benilov} reduced the nonlinear equation (\ref{model})
with $\alpha^3 = 3$ to the linear advection--diffusion equation:
\begin{equation}
\label{pde}
\frac{\partial h}{\partial t} + \frac{\partial^4 h}{\partial x^4} =
V(t) \frac{\partial h}{\partial x}, \quad x > 0, \;\; t > 0,
\end{equation}
subject to the boundary conditions
\begin{equation}
\label{bc-pde}
h |_{x = 0} = 1, \quad  h_{x} |_{x = 0} = 0, \quad h_{xxx} |_{x = 0} = - \frac{1}{2}, \quad
t \geq 0,
\end{equation}
We assume that $h, h_x, h_{xx} \to 0$ as $x \to \infty$: in fact, any constant value of
$h$ at infinity is allowed thanks to the invariance of the linear advection--diffusion
equation (\ref{pde}) with respect to the shift
and scaling transformations. Indeed, if $h(x,t)$ solves the boundary--value problem (\ref{pde})--(\ref{bc-pde}), then
$$
\tilde{h}(x,t) = H + (1-H) h(\xi,\tau), \quad \xi = \frac{x}{(1-H)^{1/3}}, \quad
\tau = \frac{t}{(1-H)^{4/3}},
$$
with constant $H < 1$ solves the same advection--diffusion equation (\ref{pde})
with the same boundary conditions (\ref{bc-pde}) but for the variable
speed $\tilde{V}(t) = \frac{V(\tau)}{1-H}$ and with the asymptotic value $h(x,t) \to H$
as $x \to \infty$.

With three boundary conditions at $x = 0$ and
the decay conditions as $x \to \infty$, the initial-value problem for equation (\ref{pde})
is over-determined and the third (over-determining) boundary condition at $x = 0$ is used
to find the dependence of $V$ on $t$. Local existence of
solutions to the boundary--value problem (\ref{pde})--(\ref{bc-pde})
was proved in our work \cite{Pelinovsky} using
Laplace transform in $x$ and the fractional power series expansion in $t$.

We shall consider the time evolution of the boundary--value problem
(\ref{pde})--(\ref{bc-pde}) starting with the initial data
$h|_{t = 0} = h_0(x)$ for a suitable function $h_0$. In particular,
we assume that the profile $h_0(x)$ decays monotonically to zero as $x \to \infty$
and that $0$ is a non-degenerate maximum of $h_0$ such that
$h_0(0) = 1$, $h_0'(0) = 0$, and $h_0''(0) < 0$. If the solution $h(x,t)$
losses monotonicity in $x$ during the dynamical
evolution, for instance, due to the value of
\begin{equation}
\label{beta}
\beta(t) := h_{xx}(0,t)
\end{equation}
crossing $0$ from the negative side,
then we say that the flow becomes non-physical for further times and
the model breaks. Simultaneously, this may mean that the velocity $V(t)$ blows up, as
it is defined for sufficiently strong solutions of
the advection--diffusion equation (\ref{pde}) by the pointwise equation:
\begin{equation}
\label{contact-equation}
h_{xxxxx}(0,t) = V(t) \beta(t),
\end{equation}
which follows by differentiation of (\ref{pde}) in $x$ and setting $x \to 0$.

The main claim of \cite{Benilov} based on numerical computations of the reduced equation
(\ref{pde}) as well as more complicated thin-film equations is that for any suitable $h_0$, there is
a finite positive time $t_0$ such that $V(t) \to -\infty$ and $\beta(t) \to 0^-$ as $t \uparrow t_0$.
Moreover, it is claimed that $V(t)$ behaves near the blowup time as the logarithmic function of $t$, e.g.
\begin{equation}
\label{numerical-claim}
V(t) \sim C_1 \log(t_0 - t) + C_2, \quad \mbox{\rm as} \quad t \uparrow t_0,
\end{equation}
where $C_1$, $C_2$ are positive constants.

The goal of this paper is to inspect possible blow-up rates of the
singularity formation in the boundary-value problem (\ref{pde})--(\ref{bc-pde}).
First, we use apriori energy estimates to show that $V(t)$ cannot
remain positive for all times for smooth solutions of the boundary--value
problem (\ref{pde})--(\ref{bc-pde}). This result implies simultaneously
two things: if $V(t)$ remains positive, the smooth solution blows up in a
finite time, and if a smooth solution exists for all times, then $V(t)$ either oscillates
or become negative. Similarly, we also show that
$\beta(t)$ and $V(t) \beta^2(t)$ cannot remain negative for all times
in the same sense: if $\beta(t)$ and $V(t) \beta^2(t)$  remain negative,
the smooth solution blows up in a finite time, and if a smooth solution exists
for all times, then $\beta(t)$ and $V(t) \beta^2(t)$  either oscillate
or become positive. Combination of both results shows that the only way a smooth
solution can exist for all times is if the variable speed $V(t)$ oscillates from
positive to negative values back and forth.

Second, we study the class of self-similar solutions based on
the scaling transformation (\ref{asymptotic-reduction}).
The class of self-similar solutions is defined by the
linear advection--diffusion equation (\ref{pde}),
the decay condition at infinity, and the first two boundary conditions
(\ref{bc-pde}). The third boundary condition $h_{xxx} |_{x = 0} = - \frac{1}{2}$
is not satisfied for the self-similar solutions and we replace this
boundary condition with new condition $h_{xxx} |_{x = 0} = \gamma_0 V(t)$ for a fixed
$\gamma_0 < 0$. We show that the solution blows up in a finite time for positive $V(t)$ and
positive $\beta(t)$, which agrees with the scaling transformation (\ref{asymptotic-reduction})
but does not correspond to the physical requirements of the flow on Figure \ref{fig-1}.

Finally, we study how $\beta(t)$ may vanish and $V(t)$ may diverge in a finite time
by using the pointwise equation (\ref{contact-equation}) and its derivative.
We find yet another rate of singularity formulations, which is different
from the rates based on the scaling transformation (\ref{asymptotic-reduction}) and
on the numerically claimed result (\ref{numerical-claim}). Therefore,
further studies of the boundary-value problem (\ref{pde})--(\ref{bc-pde})
including more precise numerical studies are required.
These studies will be reported elsewhere.

The remainder of this paper is organized as follows. Section 2 gives apriori
energy estimates for the boundary--value problem (\ref{pde})--(\ref{bc-pde}).
Section 3 describes self-similar solutions describing blow-up rate of the singularity
formulations. Section 4 reports analysis following from pointwise equations.

\section{Apriori energy estimates}

Let us consider the advection-diffusion equation (\ref{pde}) subject to the boundary conditions
(\ref{bc-pde}) and the decay condition $h, h_x, h_{xx} \to 0$ as $x \to \infty$.
We assume existence of a smooth solution to this boundary-value problem and
show that $V(t)$ cannot remain positive for all times.

\begin{theorem}
Solutions of the boundary--value problem (\ref{pde})--(\ref{bc-pde}) do not exist
in class $h \in C(\mathbb{R}_+,L^2(\mathbb{R}_+)) \cap L^2(\mathbb{R}_+,H^2(\mathbb{R}_+))$
if $V(t) \geq V_0 > -1$ for all $t \geq t_0 \geq 0$.
\end{theorem}

\begin{proof}
From the advection-diffusion equation (\ref{pde}), we have the energy balance:
$$
\partial_t \left( \frac{1}{2} h^2 \right) + \partial_x \left( h h_{xxx} - h_x h_{xx} -
\frac{1}{2} V(t) h^2 \right) + (h_{xx})^2 = 0.
$$
Integrating this equation in $x$ on $(0,\infty)$ and using the boundary conditions
(\ref{bc-pde}) and the decay conditions as $x \to \infty$, we obtain apriori energy
estimates:
\begin{equation}
\frac{d}{dt} \| h(\cdot,t) \|_{L^2(\mathbb{R}_+)}^2 + 2 \| h_{xx}(\cdot,t) \|^2_{L^2(\mathbb{R}_+)} = -(1 + V(t)).
\label{apriori-1}
\end{equation}
If we have a solution in class $h \in C(\mathbb{R}_+,L^2(\mathbb{R}_+)) \cap L^2(\mathbb{R}_+,H^2(\mathbb{R}_+))$, then
integrating the apriori energy estimate (\ref{apriori-1}), we obtain
\begin{equation}
\| h(\cdot,t) \|_{L^2(\mathbb{R}_+)}^2 + 2 \int_0^t \| h_{xx}(\cdot,\tau) \|^2_{L^2(\mathbb{R}_+)} d \tau
= \| h_0 \|_{L^2(\mathbb{R}_+)}^2 - \int_0^t (1 + V(\tau)) d \tau.
\end{equation}
Since the left-hand-side is strictly positive, the assertion of the theorem is proved.
\end{proof}

Next, we rewrite the advection--diffusion equation (\ref{pde})
for the variable $u = h_x$ in the form
\begin{equation}
\label{pde-u}
u_t + u_{xxxx} = V(t) u_x, \quad \; x > 0, \;\; t > 0,
\end{equation}
subject to the boundary conditions at the contact line
\begin{equation}
\label{bc}
u |_{x = 0} = 0, \quad  u_{xx} |_{x = 0} = -\frac{1}{2}, \quad
u_{xxx} |_{x = 0} = 0, \quad t \geq 0,
\end{equation}
where the boundary conditions $u_{xxx} |_{x = 0} = h_{xxxx} |_{x=0} = 0$ follows
from the boundary conditions $h|_{x = 0} = 1$ and $h_x |_{x =0} = 0$ as well as the
advection--diffusion equation (\ref{pde}) as $x \to 0$. Denote $\beta(t) = h_{xx} |_{x = 0} = u_x |_{x=0}$
and recall that $\beta(0) < 0$ initially. Again, we assume existence of a smooth solution
to the boundary-value problem (\ref{pde-u})--(\ref{bc}) and show that $\beta(t)$ and $V(t) \beta^2(t)$
cannot remain negative for all times.

\begin{theorem}
Solutions of the boundary--value problem (\ref{pde-u})--(\ref{bc}) do not exist
in class $u \in C(\mathbb{R}_+,L^2(\mathbb{R}_+)) \cap L^2(\mathbb{R}_+,H^2(\mathbb{R}_+))$ if
$\beta(t) \leq \beta_0 < 0$ for all $t \geq t_0 \geq 0$.
\end{theorem}

\begin{proof}
From the advection-diffusion equation (\ref{pde-u}), we have the energy balance:
$$
\partial_t \left( \frac{1}{2} u^2 \right) + \partial_x \left( u u_{xxx} - u_x u_{xx} -
\frac{1}{2} V(t) u^2 \right) + (u_{xx})^2 = 0.
$$
Integrating this equation in $x$ on $(0,\infty)$ and using the boundary conditions
(\ref{bc}) and the decay conditions as $x \to \infty$, we obtain apriori energy
estimates:
\begin{equation}
\label{apriori-2}
\frac{d}{dt} \| u(\cdot,t) \|_{L^2(\mathbb{R}_+)}^2 + 2 \| u_{xx}(\cdot,t) \|^2_{L^2(\mathbb{R}_+)} = \beta(t).
\end{equation}
If we have a solution in class $u \in C(\mathbb{R}_+,L^2(\mathbb{R}_+)) \cap L^2(\mathbb{R}_+,H^2(\mathbb{R}_+))$, then
integrating the apriori energy estimate (\ref{apriori-2}), we obtain
\begin{equation}
\| u(\cdot,t) \|_{L^2(\mathbb{R}_+)}^2 + 2 \int_0^t \| u_{xx}(\cdot,\tau) \|^2_{L^2(\mathbb{R}_+)} d \tau
= \| u_0 \|_{L^2(\mathbb{R}_+)}^2 + \int_0^t \beta(\tau) d \tau.
\end{equation}
Since the left-hand-side is strictly positive, the assertion of the theorem follows.
\end{proof}

\begin{theorem}
Solutions of the boundary--value problem (\ref{pde-u})--(\ref{bc}) do not exist
in class $u \in C(\mathbb{R}_+,H^1(\mathbb{R}_+)) \cap L^2(\mathbb{R}_+,H^3(\mathbb{R}_+))$ if
$V(t) \leq V_0 < 0$ and $|\beta(t)| \geq \beta_0 > 0$ for all $t \geq t_0 \geq 0$.
\end{theorem}

\begin{proof}
Multiplying the advection-diffusion equation (\ref{pde-u}) by $u_{xx}$,
integrating this equation in $x$ on $(0,\infty)$, and using the boundary conditions
(\ref{bc}) and the decay conditions as $x \to \infty$, we obtain apriori energy
estimates:
\begin{equation}
\label{apriori-3}
\frac{d}{dt} \| u_x(\cdot,t) \|_{L^2(\mathbb{R}_+)}^2 + 2 \| u_{xxx}(\cdot,t) \|^2_{L^2(\mathbb{R}_+)} = V(t) \beta^2(t).
\end{equation}
If we have a solution in class $u \in C(\mathbb{R}_+,H^1(\mathbb{R}_+)) \cap L^2(\mathbb{R}_+,H^3(\mathbb{R}_+))$, then
integrating the apriori energy estimate (\ref{apriori-3}), we obtain
\begin{equation}
\| u_x(\cdot,t) \|_{L^2(\mathbb{R}_+)}^2 + 2 \int_0^t \| u_{xxx}(\cdot,\tau) \|^2_{L^2(\mathbb{R}_+)} d \tau
= \| u_0' \|_{L^2(\mathbb{R}_+)}^2 + \int_0^t V(\tau) \beta^2(\tau) d \tau.
\end{equation}
Since the left-hand-side is strictly positive, the assertion of the theorem follows.
\end{proof}

\section{Self-similar solutions for singularity formations}

Let us consider the class of self-similar solutions to the
linear advection--diffusion equation (\ref{pde}):
\begin{equation}
\label{self-similar}
V(t) = \frac{t_0 V_0}{(t_0-t)^{3/4}}, \quad
h(x,t) = f(\xi), \quad \xi = \frac{x}{(t_0-t)^{1/4}},
\end{equation}
where $t_0$ is an arbitrary positive parameter for a finite blowup time,
$V_0$ is an arbitrary parameter for the initial velocity, and $f(\xi)$
is a solution of the differential equation:
\begin{equation}
\label{ode}
\frac{d^4 f}{d \xi^4} + \frac{1}{4} (\xi - 4 t_0 V_0) \frac{d f}{d \xi} = 0, \quad \xi > 0.
\end{equation}
We are looking at a solution of the boundary-value problem associated with
the first two boundary conditions at the contact line:
\begin{equation}
\label{bc-ode}
f(0) = 1, \quad f'(0) = 0,
\end{equation}
and the decay condition $f(\xi), f'(\xi) \to 0$ as $\xi \to \infty$. Note
that the third condition at the contact line $h_{xxx} |_{x = 0} = -\frac{1}{2}$
is not satisfied by the self-similar solution (\ref{self-similar}). The
revised third boundary condition is given by
\begin{equation}
\label{extended-bc}
h_{xxx} |_{x = 0} = \frac{f'''(0)}{(t_0-t)^{3/4}} = \gamma_0 V(t),
\end{equation}
where $\gamma_0$ is constant such that $f'''(0) = t_0 V_0 \gamma_0$.
Also note that the class of self-similar solutions (\ref{self-similar})
is compatible with the asymptotic scaling (\ref{asymptotic-reduction}) used in
the derivation of the linear advection-diffusion equation (\ref{pde}).

Setting
$$
g(z) = f'(\xi), \quad z = \xi - 4 t_0 V_0,
$$
we reduce the boundary-value problem (\ref{ode})--(\ref{bc-ode}) to the following system:
\begin{equation}
\label{ode-reduced}
\left\{ \begin{array}{l} 4 g'''(z) + z g(z) = 0, \quad \quad z > z_0, \\
g(z_0) = 0, \\
g(z), g'(z) \to 0 \quad \mbox{\rm as} \;\; z \to \infty, \end{array} \right.
\end{equation}
where $z_0 = -4t_0 V_0$. A suitable solution of this boundary--value problem
is constructed in the following theorem.

\begin{theorem}
There exists a unique (up to scalar multiplication) positive solution
of the boundary--value problem (\ref{ode-reduced}) on $(z_0,\infty)$ with $z_0 < 0$.
\end{theorem}

\begin{proof}
As $z \to \infty$, there are three fundamental solutions
of the linear equation
\begin{equation}
\label{diff-eq}
4 g'''(z) + z g(z) = 0.
\end{equation}
One solution decays to $0$ monotonically as $z \to \infty$ and
the other two solutions oscillate and diverge as $z \to \infty$. Therefore, the space of
solutions of the boundary--value problem (\ref{ode-reduced}) is spanned
by a particular solution (denoted by $G$) decaying to $0$ at infinity.
To define $G$ uniquely, we construct a decaying solution of
the differential equation (\ref{diff-eq}) asymptotically by
using the WKB analysis \cite{WKB}:
\begin{equation}
\label{solution-G}
G(z) = \exp\left(-\frac{3 z^{4/3}}{2^{8/3}}\right) \left[ \frac{1}{z^{1/3}} + \mathcal{O}\left(\frac{1}{z^{2/3}}\right)\right]
\quad \mbox{\rm as} \quad z \to \infty,
\end{equation}
where corrections terms can be identified in terms of power series
in inverse powers of $z^{1/3}$. The solution $G$ of the linear
equation (\ref{diff-eq}) can be extended globally for all $z \in \mathbb{R}$.
To satisfy the boundary condition at $z_0$, it remains to show
that there is $z_0 \in \mathbb{R}$ such that $G(z_0) = 0$.

It is clear that $z_0 \in \mathbb{R}$ exists. Indeed, if $z_0$ does not exist,
then $G(z)$ remains positive for all $z \in \mathbb{R}$, which is only possible
if $G(z)$ decays to $0$ monotonically as $z \to -\infty$ (the other two
solutions again oscillate and diverge as $z \to -\infty$). However, then $G \in H^2(\mathbb{R})$
is a global solution of the differential equation (\ref{diff-eq}) for all $z \in \mathbb{R}$.
Multiplying this equation by $G'$ and integrating by parts, we obtain the
contradiction
\begin{equation}
4 \int_{\mathbb{R}} (G'')^2 dz + \frac{1}{2} \int_{\mathbb{R}} G^2 dz = 0,
\end{equation}
which proves that no $G \in H^2(\mathbb{R})$ may exist. Furthermore,
$z_0 < 0$ because $G(z)$ is monotonically decaying for all $z > 0$. To see this,
we use the fact that the differential equation (\ref{diff-eq})
is invariant under the transformation
$z \mapsto -z$, so that $\tilde{G}(z) := G(-z)$ is another solution of (\ref{diff-eq}).
The function $\tilde{G}(z)$ increases monotonically for large negative $z$.
Since $4 \tilde{G}'''(z) = -z \tilde{G}(z) > 0$ for all $z < 0$,
$\tilde{G}(z)$ remains monotonically increasing for all $z \leq 0$ and hence $G(z)$ decreases
monotonically for all $z \geq 0$. Therefore, $z_0 < 0$, that is,
$V_0 > 0$ (if $t_0 > 0$). The value $z_0$ is uniquely determined as
the largest negative zero of the positive function $G(z)$.
\end{proof}

Figure \ref{fig-2} shows numerical approximation of the solution $G(z)$ satisfying the boundary--value
problem (\ref{ode}). The numerical approximation is obtained with the standard Heun method.
\begin{figure}[htbp]
\begin{center}
\includegraphics[height=5cm]{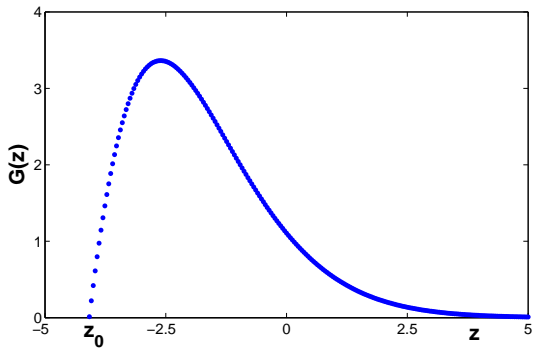}
\end{center}
\caption{Numerical approximation of the solution $G(z)$ of the boundary--value problem (\ref{ode}).}
\label{fig-2}
\end{figure}

A general solution of the boundary-value problem (\ref{ode-reduced}) is
given by $g(z) = C G(z)$. To determine the constant $C$, we use
the modified boundary condition (\ref{extended-bc}). If $h_{xxx} |_{t = 0} < 0$
as inherited from the third boundary condition in (\ref{bc-pde}), we should have
$g''(0) = f'''(0)  < 0$ or $C G''(z_0) < 0$. By the continuity arguments,
we have $G''(z_0) < 0$ and therefore, $C > 0$. Indeed, $\tilde{G}(z) := G(-z)$ is monotonically
increasing function for all $z \leq 0$ with $4 \tilde{G}'''(z) = -z \tilde{G}(z) > 0$.
When $z > 0$, $\tilde{G}'''(z) < 0$ as long as $\tilde{G}(z) > 0$, so that
there is $\tilde{z}_0 \in (0,z_0)$, such that $\tilde{G}''(z) < 0$ for all
$z \in (-\tilde{z}_0,-z_0]$, or equivalently, $G''(z) < 0$ for all
$z \in [z_0,\tilde{z}_0)$. Therefore, $G''(z_0) < 0$ as on Figure \ref{fig-2}.

By the same argument, there is $\tilde{\tilde{z}}_0 \in (\tilde{z}_0,z_0]$
such that $\tilde{G}'(z) < 0$ for all $z \in (-\tilde{\tilde{z}}_0,-z_0]$,
or equivalently, $G'(z) > 0$ for all $z \in [z_0,\tilde{\tilde{z}}_0)$.
Therefore, $G'(z_0) > 0$ as on Figure \ref{fig-2}, which implies that
$$
\beta(t) = \frac{f''(0)}{(t_0-t)^{1/2}} = \frac{C G'(z_0)}{(t_0-t)^{1/2}} > 0, \quad t \in [0,t_0).
$$

We conclude that the class of self-similar solutions (\ref{self-similar}) does not
represent the relevant dynamics of the boundary-value problem (\ref{pde})--(\ref{bc-pde})
in the context of the physical requirements of the flow on Figure \ref{fig-1}
because $\beta(t)$ is supposed to be negative at least for initial values of $t \geq 0$.

\section{Pointwise equations}

We give here additional estimates of how the solution of the boundary--value problem
(\ref{pde})--(\ref{bc-pde}) may blow up in a finite time, based on the pointwise
equation (\ref{contact-equation}) and its derivative.
We look at the boundary--value problem (\ref{pde-u})--(\ref{bc}) and assume existence
of a sufficiently smooth solution. By taking the limit $x \to 0$, we recover the pointwise
equation (\ref{contact-equation}) rewritten in new variables as
\begin{equation}
\label{contact-1}
u_{xxxx} |_{x = 0} = V(t) \beta(t), \quad t \geq 0,
\end{equation}
where $\beta(t) = u_x |_{x = 0}$. By taking a derivative of the linear advection--diffusion
equation (\ref{pde-u}) in $x$ and the limit $x \to 0$, we obtain another pointwise equation:
\begin{equation}
\label{contact-2}
\frac{d \beta}{d t} + u_{xxxxx} |_{x = 0} = -\frac{1}{2} V(t), \quad t \geq 0.
\end{equation}
The system of equations (\ref{contact-1}) and (\ref{contact-2}) can be rewritten in
the partially closed form:
\begin{equation}
\label{contact}
\frac{d \beta}{d t} = -\frac{u_{xxxx} |_{x = 0}}{2 \beta(t)} - u_{xxxxx} |_{x = 0}, \quad t \geq 0.
\end{equation}

Let us now assume that there is $t_0 > 0$ such that
\begin{equation}
\label{assumed-rate}
\beta(t) \to 0, \quad u_{xxxx} |_{x = 0} \to a_4, \quad u_{xxxxx} |_{x = 0} \to a_5, \quad \mbox{\rm as} \quad t \uparrow t_0,
\end{equation}
where $a_4 \neq 0$ and $|a_5| < \infty$. Then, asymptotic analysis of the differential
equation (\ref{contact}) shows that
\begin{equation}
\label{asymptotic-rate}
\beta^2(t) = a_4(t_0 - t) + \mathcal{O}(t_0 - t)^{3/2}, \quad
V(t) = \sqrt{\frac{a_4}{t_0 - t}} + \mathcal{O}(1), \quad \mbox{\rm as} \quad t \uparrow t_0,
\end{equation}
under the constraint that $a_4 > 0$. The asymptotic rate (\ref{asymptotic-rate})
is different both from the scaling transformation (\ref{asymptotic-reduction}) and
the numerically claimed result (\ref{numerical-claim}). In the context of the
numerical result (\ref{numerical-claim}), this pointwise analysis may imply that
either $a_4 = 0$ or $a_5 \to \infty$ in the assumption (\ref{assumed-rate}).

We conclude that three different rates of the singularity formations
claimed in (\ref{numerical-claim}) and obtained in (\ref{self-similar}) and (\ref{asymptotic-rate})
indicate complexity of dynamics of the boundary-value problem (\ref{pde})--(\ref{bc-pde})
or its equivalent version (\ref{pde-u})--(\ref{bc}). Further studies of
dynamical evolution of contact lines within this reduced problem are needed,
including more precise numerical simulations.

\vspace{0.25cm}

{\bf Acknowledgement:} The authors are thankful to E.S. Benilov and R. Taranets
for useful discussions and for sharing their unpublished results at an early stage of his research.

\end{document}